\RequirePackage[english]{babel}
\documentclass[a4paper,twoside,11pt]{amsart}
\usepackage[latin1]{inputenc}
\usepackage{amsmath}
\usepackage{amssymb}
\usepackage{amsfonts}
\usepackage{tikz}
\usepackage{fixltx2e}[2005/12/01]

\usepackage{enumerate}

\usepackage[matrix,arrow,cmtip,frame,graph,curve,2cell]{xy} % curve, all
\SelectTips{cm}{}
\newdir{(}{{}*!/-5pt/@^{(}}
\newdir{(x}{{}*!/-5pt/@_{(}}
\newdir{+}{{}*!/-9pt/{}}
\newdir{>+}{@{>}*!/-9pt/{}}
\entrymodifiers={+!!<0pt,\fontdimen22\textfont2>}
\UseAllTwocells

\usepackage[nosubsections]{mytheorems2}

\newcommand\pref[1]{\eqref{#1}}

\newcommand\inj{\hookrightarrow}
\newcommand\iso{\cong} % Isomorphism
\newcommand\map[3]{#1\colon #2\rightarrow #3}
\newcommand\injmap[3]{#1\colon #2\hookrightarrow #3}
\newcommand\id[1]{\mathrm{id}_{#1}} % Identity map
\DeclareMathOperator{\Hom}{Hom}
\newcommand{\catHom}{\mathbf{Hom}} % Hom-category (groupoid)

\newcommand\sO{\mathcal{O}}

\newcommand\Spec{\mathrm{Spec}}
\newcommand\A[1]{\mathbb{A}^{#1}}    % Affine space
\renewcommand\P[1]{\mathbb{P}^{#1}}    % Projective space

\newcommand\Sch{\mathbf{Sch}}
\newcommand\Set{\mathbf{Set}}
\newcommand\Setp{\mathbf{Set}_{*}}

\newcommand{\etale}{\'{e}tale}

\newcommand\red{\mathrm{red}}
\newcommand{\et}{\text{\'et}} % etale subscript

\begin{document}

\title[Canonical embedding of an unramified morphism]{The canonical embedding of an unramified morphism in an \'etale morphism}
\author{David Rydh}
%\address{Department of Mathematics, KTH, 100 44 Stockholm, Sweden}
\address{Department of Mathematics, University of California, Berkeley,
970 Evans Hall \#3840, Berkeley, CA 94720-3840 USA}
\thanks{Supported by the Swedish Research Council.}
\email{dary@math.berkeley.edu}
\date{2010-05-13}

%14A20 Foundations - Generalizations (algebraic spaces, stacks)

\begin{abstract}
We show that every unramified morphism $X\to Y$ has a canonical and universal
factorization $X\inj E_{X/Y}\to Y$ where the first morphism is a closed
embedding and the second is \etale{} (but not separated).
\end{abstract}

\maketitle

\begin{section}{Introduction}
It is well-known that any unramified morphism $\map{f}{X}{Y}$ of schemes (or
Deligne--Mumford stacks) is an \etale{}-local embedding, i.e., there exists a
commutative diagram
{\def\theequation{*}%
\begin{equation}\label{E:local-embedding}
\vcenter{\xymatrix{X'\ar[r]^{f'}\ar[d] & Y'\ar[d]\\
X\ar[r]^f & Y\ar@{}[ul]|\circ}}
\end{equation}}%
where $f'$ is a closed embedding and the vertical morphisms are \etale{} and
surjective. To see this, take \etale{} presentations $Y'\to Y$ and $X'\to
X\times_Y Y'$ such that $X'$ and $Y'$ are schemes and then
apply~\cite[Cor.~18.4.7]{egaIV}. This proof utterly fails if $Y$ is a stack
which is not Deligne--Mumford and the existence of a
diagram~\eqref{E:local-embedding} appears to be unknown in this case.
Also, if we require $Y'\to Y$ to be separated, then in general there is no
\emph{canonical} choice of the diagram~\eqref{E:local-embedding}.

The purpose of this article is to show that for an arbitrary unramified
morphism of algebraic stacks, there is
% always
a \emph{canonical} \etale{} morphism $E_{X/Y}\to Y$ and a closed embedding
$X\inj E_{X/Y}$ over $Y$. If $\map{f}{X}{Y}$ is an unramified morphism of
\emph{schemes} (or algebraic spaces), then $E_{X/Y}$ is an algebraic space.

\begin{remark}\label{R:immersion}
If $\map{f}{X}{Y}$ is an \emph{immersion}, then there is a canonical
factorization $X\inj U\to Y$ where $X\inj U$ is a closed immersion and $U\to Y$
is an open immersion. Here $U$ is the largest open neighborhood of $X$ such
that $X$ is closed in $U$. Explicitly, $U=Y\setminus (\overline{X}\setminus
X)$. This factorization commutes with flat base change
if $f$ is quasi-compact but not with arbitrary base change
unless $f$ is a closed immersion.
% Counter-example: (f has to be quasi-compact)
% Take $Y=(A^1)^\cons$ and let $X$ be any infinite subset of open
% points of $Y$. Then $U$ is the union of all open points, i.e., the
% complement of the ``generic'' non-open point.  After localizing at
% the ``generic'' point, we get $X=U=\emptyset$ so we can enlarge $U=Y$
% so the factorization does not commute with flat base change.
% In this example $f$ is even an open immersion! (and we can take $X=U$
% as the complement of the generic point if we like)
%
% A slightly less trivial-looking example is the following:
% Let $Y=Spec(k[x_1,x_2,...]$, $U=Y\setminus \{0\}$ and $X\inj U$ defined
% by $(x_1,x_2,...,x_{i-1},x_i-1,x_{i+1},...)$ over $x_i\neq 0$. Then
% $f$ is an immersion which is not quasi-compact. The factorization does
% not commute with the localization at the origin.
The canonical factorization that
we will construct is slightly different and commutes with arbitrary base change
but is not separated. For an immersion $\map{f}{X}{Y}$, the scheme $E_{X/Y}$ is
the gluing of $U$ and $Y$ along the open subsets $U\setminus X=Y\setminus
\overline{X}$.
\end{remark}

\begin{theorem}\label{T:main-theorem}
Let $\map{f}{X}{Y}$ be an unramified morphism of algebraic stacks. Then there
exists an \etale{} morphism $\map{e=e_f}{E_{X/Y}}{Y}$ together with a closed
immersion $\injmap{i=i_f}{X}{E_{X/Y}}$ and an open immersion
$\map{j=j_f}{Y}{E_{X/Y}}$ such that $f=e\circ i$, $\id{Y}=e\circ j$ and the
complement of $i(X)$ is $j(Y)$. We have that:
\begin{enumerate}
\item The triple $(e,i,j)$ is unique up to unique $2$-isomorphism, i.e.,
if $\map{e'}{E'}{Y}$ is an \etale{} morphism, $\injmap{i'}{X}{E'}$ is a closed
immersion and $\map{j'}{Y}{E'}$ is an open immersion over $Y$ such that the
complement of $i'(X)$ is $j'(Y)$, then there is an isomorphism
$\map{\varphi}{E'}{E_{X/Y}}$ such that $e'=e\circ\varphi$, $i=\varphi\circ i'$
and $j=\varphi\circ j'$, and $\varphi$ is unique up to unique $2$-isomorphism.
\label{TI:uniqueness}
\item Let $\map{g}{Y'}{Y}$ be any morphism and let $\map{f'}{X'}{Y'}$ be the
pull-back of $f$ along $g$. Then the pull-backs of $e_f$, $i_f$ and $j_f$ along
$g$ coincide with $e_{f'}$, $i_{f'}$ and $j_{f'}$.
%
%For every morphism $Y'\to Y$ there is a canonical isomorphism
%$(E_{X/Y})\times_Y Y'\to E_{X\times_Y Y'/Y'}$.
\label{TI:base-change}
\item $e$ is an isomorphism if and only if $X=\emptyset$.
\label{TI:iso}
\item $e$ is separated if and only if $f$ is \etale{} and separated.
\label{TI:sep}
\item $e$ is universally closed (resp.\ quasi-compact, resp.\ representable) if
and only if $f$ is so. In particular, $e$ is universally closed, quasi-compact
and representable if $f$ is finite.
\label{TI:uc,qc,repr}
\item $e$ is of finite presentation (resp.\ quasi-separated) if and only if $f$
is of constructible finite type (resp.\ quasi-separated and locally of
constructible finite type).
For the definition of the latter notions, see Appendix~\ref{A:constructible}.
\label{TI:qs,finite-pres}
\item $e$ is a local isomorphism if and only if $f$ is a local immersion.
\label{TI:loc-immersion}
\item If $\map{g}{V}{X}$ is an \etale{} morphism, then there exists a unique
\etale{} morphism $\map{g_*}{E_{V/Y}}{E_{X/Y}}$ such that the pull-back of
$i_{f}$ (resp.\ $j_{f}$) along $g_*$ is $i_{f\circ g}$ (resp.\ $j_{f\circ
g}$). If $g$ is surjective (resp.\ representable, resp.\ an open immersion),
then so is $g_*$.
\label{TI:etale-pushfwd}
\item If $\map{g}{V}{X}$ is a closed immersion then there is a natural
surjective morphism $\map{g^*}{E_{X/Y}}{E_{V/Y}}$ such that $i_{f\circ
g}=g^*\circ i_f\circ g$ and $j_{f\circ g}=g^*\circ j_f$. The morphism $g^*$ is
an isomorphism if and only if $g$ is a nil-immersion (i.e., a bijective closed
immersion).
If $g$ is an open and closed immersion, then $g^*g_*=\id{E_{V/Y}}$.
%
% $g^*$ is natural but not unique. We obtain uniqueness if we require that
% $g^*=j_{f\circ g}\circ e_f$ over the open subset
% $E_{X/Y}\setminus i_f(g(V))$.
\label{TI:closed-pullback}
\end{enumerate}
\end{theorem}

We call the \etale{} morphism $\map{e}{E_{X/Y}}{Y}$ the \emph{\etale{}
envelope} of $X\to Y$. Note that the fibers of $e$ coincide with the fibers of
$X\amalg Y\to Y$. In Definition~\pref{D:etale-envelope:repr}
(resp.~\pref{D:etale-envelope:general}) we give a functorial description of
$E_{X/Y}$ in the representable (resp.\ general) case.

For the definitions of representable and unramified morphisms of stacks, see
Appendices~\ref{A:stacks} and~\ref{A:unramified-and-etale}. If the reader does
not care about stacks, then rest assured that any scheme (or algebraic space)
is an algebraic stack and that any morphism of schemes (or algebraic spaces) is
representable. For schemes (or algebraic spaces), unique up to unique
$2$-isomorphism means unique up to unique isomorphism.

\begin{remark}
Even if $\map{f}{X}{Y}$ is a morphism of schemes (as is the case if $Y$ is a
scheme and $f$ is representable and separated), it is often the case that
$E_{X/Y}$ is not a scheme but an algebraic space, cf.\
Example~\pref{E:alg-space}. However, if $f$ is a local immersion, then
$E_{X/Y}$ is a scheme by~\ref{TI:loc-immersion}.
\end{remark}

\begin{remark}
For any representable morphism $\map{f}{X}{Y}$ locally of finite type
one can define a natural operation
$\map{f_{\#}}{\Setp(X)}{\Setp(Y)}$ on \etale{} sheaves of
\emph{pointed sets} such that if $f$ is unramified, then the \etale{}
envelope $E_{X/Y}$ is the sheaf $f_{\#}\{0,1\}_X$. Here $\{0,1\}_X$
denotes the constant sheaf of a pointed set with two elements. If $f$
is \etale{}, then $f_{\#}$ is left adjoint to the pull-back $f^{-1}$
of pointed sets and if $f$ is a monomorphism, then $f_{\#}=f_{!}$ is
extension by zero. We do not develop the general theory of $f_{\#}$ in
this article.
\end{remark}

\begin{remark}
Note that ``quasi-compact'' is equivalent to ``finite type'' for unramified
morphisms. When $Y$ is non-noetherian, the question of finite presentation (or
equivalently of quasi-separatedness) of $E_{X/Y}\to Y$ is somewhat delicate,
cf.\ Appendix~\ref{A:constructible}.
\end{remark}

%\begin{remark}
%Let $\map{f}{X}{Y}$ be a quasi-compact representable unramified morphism such
%that $X$ is \etale{} onto its schematic image $Z=f(X)$, e.g., let $f$ be a
%quasi-compact immersion. Let $\injmap{k}{Z}{Y}$ be the inclusion of $Z$ and
%let $E\to Y$ be the \etale{} morphism which represents the \etale{} sheaf
%$k_*\underline{X}_{/Z}$, where $\underline{X}_{/Z}$ is the \etale{} sheaf
%on $Z$ represented by the \etale{} morphism $X\to f(X)=Z$. Then there is
%a factorization $X\inj E\to Y$ of $f$. If $f$ is an immersion, then
%$E=U$ as described in Remark~\pref{R:immersion}.
%\end{remark}

%\begin{remark}
%Fix a scheme $Y$ and let $\catC$ be the category where the objects are
%representable unramified morphisms $X\to Y$ and the morphisms are closed
%immersions $X_1\to X_2$. Let $\catD$ be the category of representable
%\etale{} morphisms $E\to Y$ with a distinguished section. The latter category
%is equivalent with the category of sheaves of pointed sets in the small
%\etale{} site. The assignment $X\to E_{X/Y}$ defines a contravariant functor
%$\catC\to \catD$. This functor localizes all nil-immersions and hence induces
%an equivalence between the category of unramified morphisms up to topology and
%the category of \etale{} morphisms with a section.
%\end{remark}

We begin with a few examples of the \etale{} envelope in
Section~\ref{S:examples}. The proof of Theorem~\pref{T:main-theorem} in the
representable case is given in Section~\ref{S:representable-case} and the
general case is dealt with in Section~\ref{S:general-case}. Some applications
of the main theorem are outlined in Section~\ref{S:applications}.
In Appendix~\ref{A:stacks} we give precise meanings to ``algebraic
space'', ``algebraic stack'' and ``representable''. In
Appendix~\ref{A:unramified-and-etale} we define unramified and \etale{}
morphisms of stacks and establish their basic properties.
Some limit results used in the non-noetherian case are given in
Appendix~\ref{A:limits}.
Finally, in Appendix~\ref{A:constructible} we define the technical condition
``of constructible finite type'' which is only used to give a characterization
of the unramified morphisms having a finitely presented \etale{} envelope in the
non-noetherian case.

Theorem~\pref{T:main-theorem} was inspired by a similar result recently
obtained by Anca and Andrei
Musta\c{t}\v{a}~\cite{mustata-mustata_local-embeddings}. They study the case
when $\map{f}{X}{Y}$ is a finite unramified morphism between proper integral
noetherian Deligne--Mumford stacks and construct a stack $F_{X/Y}$ such that
$F_{X/Y}\to Y$ is \etale{} and universally closed and such that
$F_{X/Y}\times_Y f(X)$ is a union of closed substacks $\{F_i\}$ which admit
\etale{} and universally closed morphisms $F_i\to X$. The stack $F_{X/Y}$ has
an explicit groupoid description but a functorial interpretation is missing. In
general, $F_{X/Y}$ is different from $E_{X/Y}$ and does not commute with
arbitrary base change.
\end{section}

% -------------------------------------------------------------------

\begin{section}{Examples}\label{S:examples}
\begin{example}
If $\map{f}{X}{Y}$ is \etale{}. Then $E_{X/Y}=X\amalg Y$.
\end{example}

\begin{example}
Let $Y$ be a scheme and let $X=\coprod_{i=1}^n X_i$ be the disjoint union of
closed subschemes $X_i\inj Y$. Then $E_{X/Y}$ is a scheme and can be described
as the gluing of $n+1$ copies of $Y$ as follows. Let $Y_i=Y$ for
$i=1,\dots,n$. Glue each $Y_i$ to $j(Y)=Y$ along $Y\setminus X_i$. The
resulting scheme is $E_{X/Y}$. Note that $Y_i\cap Y_j=Y\setminus (X_i\cup
X_j)$.
\end{example}

\begin{example}\label{E:node}
The following example is a special case of the previous example. Let
$Y=\Spec(k[x,y]/xy)$ be the union of the two coordinate axes in the affine
plane and let $X=\A{1}\amalg \A{1}$ be the normalization of $Y$. Then $E_{X/Y}$
can be covered by three affine open subsets isomorphic to $Y$. If we denote
these three subsets by $j(Y),Y_1,Y_2$, then $j(Y)\cap Y_1$ is the open subset
$y\neq 0$, $j(Y)\cap Y_2$ is the open subset $x\neq 0$ and $Y_1\cap
Y_2=\emptyset$.
\end{example}

\begin{example}\label{E:nodal-cubic}
Let $Y$ be a nodal cubic curve in $\P{2}$ and let $\map{f}{X}{Y}$ be the
normalization. Let $0\in Y$ be the node and let $\{+1,-1\}\subseteq X$ be its
preimage. The scheme $E_{X/Y}$ has two irreducible components $X$ and
$\overline{j(Y)}$ and $\overline{j(Y)}$ is isomorphic to the gluing of $Y$ with
$X$ along $Y\setminus \{0\}$ and $X\setminus \{+1,-1\}$. The scheme $E_{X/Y}$
is covered by two open separated subschemes $j(Y)$ and $U$. The open subset
$U=X_1\cup X_2$ is the union of two copies of $X$, the first is $i(X)$ and the
second is $\overline{j(Y)}\setminus \{0\}$, such that $\pm 1\in X_1$ is
identified with $\mp 1\in X_2$. The intersection of $j(Y)$ and $U$ is
$j(Y)\setminus 0=X_2\setminus \{+1,-1\}$.
\end{example}

\begin{example}\label{E:alg-space}
Let $Y$ be an irreducible scheme, let $Z\inj Y$ be an irreducible closed
subscheme, $Z\neq Y$, and let $\map{g}{X}{Z}$ be a non-trivial \etale{} double
cover. Then
$E_{X/Y}$ is an algebraic space which is not a scheme. In fact, let
$E=E_{X/Y}\setminus j(Z)$. Then $E\subseteq E_{X/Y}$ is open and
$\map{e|_E}{E}{Y}$ is universally closed and such that $e|_E$ is an
isomorphism outside $Z$ and coincides with $g$ over $Z$. If $\xi$ is the
generic point of $Z$, then $E_\xi=\{\eta\}$ where $\eta$ is the generic point
of $X$. If $E$ was a scheme, then $E\times_Y
\Spec(\sO_{Y,\xi})$ would be a local scheme with closed point $\eta$ and in
particular separated. This would imply that $E\times_Y \Spec(\sO_{Y,\xi})\to
\Spec(\sO_{Y,\xi})$ is finite and \etale{}.
% since $E\to Y$ is universally closed.
But $E\to Y$ has generic rank $1$ and special rank $2$.
\end{example}

% pgf-manual p 136
\tikzset{graph norm/.style={line width=1 pt}}
\tikzset{graph border/.style={line width=0.4 pt}}
\tikzset{graph over/.style={double distance=1 pt, line width=0.8 pt, draw=white,double=black}}

\noindent
\begin{tikzpicture}[scale=1.9]
% \begin{scope}[xshift=2cm]
 % Central->bottom y=0
 \draw[graph norm] plot[variable=\t,smooth,domain=-1:1] ({\t},{0.5*\t-0.1*exp(-5.0 * \t^2)^8});
 % Bottom x=0
 \draw[graph norm,red] plot[variable=\t,domain=-1:1] ({\t},{-0.5*\t-0.1});
 % Central y=0 (partly overwrite)
 \draw[graph over] plot[variable=\t,domain=-0.15:0] ({\t},{0.5*\t});
 \draw[graph norm] plot[variable=\t,domain=-1:1] ({\t},{0.5*\t});
 % Central->top x=0 (overwrite)
 \draw[graph over] plot[variable=\t,smooth,domain=-1:1] ({\t},{-0.5*\t+0.1*exp(-5.0 * \t^2)^8});
 % Central x=0 (partly overwrite)
 \draw[graph over] plot[variable=\t,domain=0.05:0.15] ({\t},{-0.5*\t});
 \draw[graph norm] plot[variable=\t,domain=-1:1] ({\t},{-0.5*\t});
 % Top y=0 (partly overwrite)
 \draw[graph over,double=red] plot[variable=\t,domain=-1:-0.05] ({\t},{0.5*\t+0.1});
 \draw[graph norm,red] plot[variable=\t,domain=-1:1] ({\t},{0.5*\t+0.1});
% \end{scope}
 \node [red,above,left] at (0.45,0.5) {$i(X)$};
 \node [right] at (0.3,0) {$\overline{j(Y)}$};
 \node [below, text centered] at (0,-1) {Example~\pref{E:node}};

%% Flipped figure
%%  % Central->bottom x=0
%%  \draw[graph norm] plot[variable=\t,smooth,domain=-1:1] ({\t},{-0.5*\t-0.1*exp(-5.0 * \t^2)^6});
%%  % Bottom y=0
%%  \draw[graph norm,red] plot[variable=\t,domain=-1:1] ({\t},{0.5*\t-0.1});
%%  % Central x=0 (partly overwrite)
%%  \draw[graph over] plot[variable=\t,domain=0:0.15] ({\t},{-0.5*\t});
%%  \draw[graph norm] plot[variable=\t,domain=-1:1] ({\t},{-0.5*\t});
%%  % Central->top y=0 (overwrite)
%%  \draw[graph over] plot[variable=\t,smooth,domain=-1:1] ({\t},{0.5*\t+0.1*exp(-5.0 * \t^2)^6});
%%  % Central y=0 (partly overwrite)
%%  \draw[graph over] plot[variable=\t,domain=-0.15:-0.05] ({\t},{0.5*\t});
%%  \draw[graph norm] plot[variable=\t,domain=-1:1] ({\t},{0.5*\t});
%%  % Top x=0 (partly overwrite)
%%  \draw[graph over,double=red] plot[variable=\t,domain=0.05:1] ({\t},{-0.5*\t+0.1});
%%  \draw[graph norm,red] plot[variable=\t,domain=-1:1] ({\t},{-0.5*\t+0.1});
%% % \end{scope}
%%  \node [red,above,right] at (-0.65,0.5) {$i(X)$};
%%  \node [left] at (-0.4,0) {$\overline{j(Y)}$};
%%  \node [below, text centered] at (0,-1) {Example~\pref{E:node}};

% y^2=x^3+x^2 :  x=t^2-1, y=t^3-t  (red version slightly different)
\begin{scope}[xshift=2.2cm]
 % Central->bottom
 \draw[graph norm] plot[variable=\t,smooth,domain=0:1.33] ({\t^2-1},{0.5*(\t^3-\t)-0.1*exp(-5.0 * (\t^2-1)^2)^10});
 % Bottom
 \draw[graph norm,red] plot[variable=\t,smooth,domain=-1.3:0.1] ({\t^2-0.9},{0.5*(\t^3-0.9*\t)-0.1*(-\t)});
 % Central first (partly overwrite)
 \draw[graph over] plot[variable=\t,smooth,domain=0.92:1] ({\t^2-1},{0.5*(\t^3-\t)});
 \draw[graph norm] plot[variable=\t,smooth,domain=-0.1:1.33] ({\t^2-1},{0.5*(\t^3-\t)});
 % Central->top
 \draw[graph over] plot[variable=\t,smooth,domain=-1.33:0] ({\t^2-1},{0.5*(\t^3-\t)+0.1*exp(-5.0 * (\t^2-1)^2)^10});
 % Central second (partly overwrite)
 \draw[graph over] plot[variable=\t,smooth,domain=-1.02:-1.05] ({\t^2-1},{0.5*(\t^3-\t)});
 \draw[graph norm] plot[variable=\t,smooth,domain=-1.33:0.1] ({\t^2-1},{0.5*(\t^3-\t)});
 % Top (partly overwrite)
 \draw[graph over] plot[variable=\t,smooth,domain=0:0.92] ({\t^2-0.9},{0.5*(\t^3-0.9*\t)+0.1*(\t)});
 \draw[graph norm,red] plot[variable=\t,smooth,domain=-0.1:1.3] ({\t^2-0.9},{0.5*(\t^3-0.9*\t)+0.1*(\t)});
 \node [red,above,left] at (0.45,0.5) {$i(X)$};
 \node [right] at (0.3,0) {$\overline{j(Y)}$};
 \node [below, text centered] at (0,-1) {Example~\pref{E:nodal-cubic}};
\end{scope}

\begin{scope}[xshift=4.4cm]
 % Frame
% \draw[graph border] (-0.8,-0.5) -- (0.8,-0.5) -- (0.8,0.5) -- (-0.5,0.8) -- cycle;
  \draw[graph border] plot[variable=\t,samples=4,domain=20:380] ({1.2*cos(\t)},{sin (\t)*0.8});
  \draw[graph border] plot[variable=\t,smooth,domain=0:360] ({0.5*cos(\t)},{0.5*sin(\t)*0.6});
  \draw[graph over,double=red] plot[variable=\t,smooth,domain=0:360] ({0.5*cos(\t-110)},{0.5*sin(\t-110)*0.6+0.2-0.1*cos(\t/2)});
  \draw[graph over,double=red] plot[variable=\t,smooth,domain=360:720] ({0.5*cos(\t-110)},{0.5*sin(\t-110)*0.6+0.2-0.1*cos(\t/2)});

%  \draw[graph norm,red] plot[variable=\t,smooth,domain=0:360] ({0.5*cos(\t-110)},{0.5*sin(\t-110)*0.6-0.1*cos(\t/2)});
%  \draw[graph over] plot[variable=\t,smooth,domain=0:360] ({0.5*cos(\t)},{0.5*sin(\t)*0.6});
%  \draw[graph over,double=red] plot[variable=\t,smooth,domain=360:720] ({0.5*cos(\t-110)},{0.5*sin(\t-110)*0.6-0.1*cos(\t/2)});

 \node [red,right] at (0.5,0.15) {$i(X)$};
 \node [below] at (0.3,-0.35) {$j(Y)$};
 \node [left, thin] at (-0.45,-0.1) {$j(Z)$};
 \node [below, text centered] at (0,-1) {Example~\pref{E:alg-space}};
\end{scope}
\end{tikzpicture}

\end{section}

%%%%%%%%%%%%%%%%%%%%%%%%%%%%%%%%%%%%%%%%%%%%%%%%%%%%%%%%%%%%%%%%%%%%%%%%%%%%%

\begin{section}{The representable case}\label{S:representable-case}
In this section we prove Theorem~\pref{T:main-theorem} for representable
unramified morphisms.

\begin{definition}\label{D:etale-envelope:repr}
Let $\map{f}{X}{Y}$ be an unramified morphism of algebraic spaces. We define a
contravariant functor $\map{E_{X/Y}}{\Sch_{/Y}}{\Set}$ as follows. For any
scheme $T$ and morphism $T\to Y$, we let $E_{X/Y}(T)$ be the set of commutative
diagrams
$$\xymatrix{& X\times_Y T\ar[d]^{\pi_2}\\
W\ar@{(->}[r]\ar[ur] & T}$$
such that $W\to X\times_Y T$ is an open immersion and $W\to T$ is a closed
immersion. Pull-backs are defined by pulling back such diagrams.
\end{definition}

The presheaf $E_{X/Y}$ is a presheaf of \emph{pointed sets}. The distinguished
element of $E_{X/Y}(T)$ is given by $W=\emptyset$.
It is also naturally a presheaf in partially ordered sets and if $f$ is
separated, then any two elements $W_1,W_2\in E_{X/Y}(T)$ have a greatest lower
bound given by $W_1\cap W_2$.

By fpqc-descent of open subsets and of closed immersions, we have that
$E_{X/Y}$ is a sheaf in the fpqc topology. Let $E_{X/Y,\et}$ denote the
restriction of $E_{X/Y}$ to the small \etale{} site on $Y$ so that
$E_{X/Y,\et}$ is an \etale{} sheaf. The first goal is to show that $E_{X/Y}$ is
\emph{locally constructible}, i.e., that $E_{X/Y}$ is the extension of
$E_{X/Y,\et}$ to the big \etale{} site.

\begin{lemma}\label{L:etale-envelope-is-loc-of-fp}
The functor $E_{X/Y}$ is \emph{locally of finite presentation}, i.e., for every
inverse limit of affine schemes $T=\varprojlim T_\lambda$ over $Y$ we have that
$$\varinjlim_{\lambda} E_{X/Y}(T_\lambda)\to E_{X/Y}(T)$$
is bijective.
\end{lemma}
\begin{proof}
An element of $E_{X/Y}(T)$ is an open immersion $\map{w}{W}{X\times_Y T}$ such
that $\map{\pi_2\circ w}{W}{T}$ is a closed immersion. As $w$ is locally of
finite presentation and $W$ is affine, there is by
Proposition~\pref{P:semi-std-limits} an \etale{} morphism
$\map{w_\lambda}{W_\lambda}{X\times_Y T_\lambda}$ such that $W_\lambda$ is
quasi-compact and quasi-separated and the pull-back of $w_\lambda$ along $T\to
T_\lambda$ is $w$. After increasing $\lambda$ we can also assume that the
morphism $\map{\pi_2\circ w_\lambda}{W_\lambda}{T_\lambda}$ is a closed
immersion by Proposition~\pref{P:limit-of-ft+qs}. Then $w_\lambda$ is an
\etale{} monomorphism and hence an open immersion. The open immersion
$w_\lambda$ determines an element of $E_{X/Y}(T_\lambda)$ which maps to $w$ so
the map in the lemma is surjective.

That the map is injective follows immediately from~\cite[Thm.~8.8.2 (i)]{egaIV}
since if $\map{w_\lambda}{W_\lambda}{X\times_Y T_\lambda}$ is an object of
$E_{X/Y}(T_\lambda)$ then $W_\lambda$ is quasi-compact and quasi-separated and
$w_\lambda$ is locally of finite presentation.
\end{proof}

The following lemma is well-known for separated unramified morphisms.

\begin{lemma}\label{L:unramified-over-strictly-henselian}
Let $S=\Spec(A)$ be the spectrum of a strictly henselian local ring with closed
point $s$, let $X$ be an algebraic space and let $X\to S$ be an unramified
morphism.
\begin{enumerate}
\item Let $\map{x}{\Spec(k(s))}{X_s}$ be a point in the closed fiber.
Then the henselian local scheme $X(x):=\Spec(\sO_{X,x})$ is an open subscheme
of $X$ and $X(x)\to S$ is a closed immersion. In particular, $X=X_1\cup X_2$ is
a union of open subspaces where $X_1$ is a scheme and $X_2\cap X_s=\emptyset$.
\item There is a one-to-one correspondence between points of $|X_s|$ and
non-empty open subspaces $W\subseteq X$ such that $W\to S$ is a closed
immersion. This correspondence takes $x\in |X_s|$ to $X(x)\subseteq X$ and
$W\subseteq X$ to $W\cap |X_s|$.
\end{enumerate}
\end{lemma}
\begin{proof}
Let $V\to X$ be an \etale{} presentation with $V$ a separated scheme and choose
a lifting $\map{v}{\Spec(k(s))}{V_s}$ of $x$. Then $V_1=\Spec(\sO_{V,v})\iso
X(x)$ is an open and closed neighborhood of $v$ and $V_1\to S$ is finite and
hence a closed immersion. It follows that $X(x)\iso V_1\to X$ is an open
immersion. The second statement follows immediately from the first.
%
% To see that W |-> W\cap X_s is injective, assume that W_1 and W_2 are two
% different neighborhoods of x such that W_i->S are closed immersions. Then
% W_1\cap W_2 is a third neighborhood so we can assume that W_1 is strictly
% contained in W_2. Then W_1\subseteq W_2 is open and closed but this
% contradicts the fact that W_2->S is a closed immersion.
\end{proof}

\begin{lemma}\label{L:points-of-small-E}
Let $\map{f}{X}{Y}$ be an unramified morphism of algebraic spaces and let
$\overline{y}\to Y$ be a geometric point. The stalk
$(E_{X/Y,\et})_{\overline{y}}$ equals $|X_{\overline{y}}|\cup \{\emptyset\}$
where $|X_{\overline{y}}|$ is the underlying set of the geometric fiber
$X_{\overline{y}}=X\times_Y \Spec(k(\overline{y}))$.
% i.e., the set of geometric points of $X_y$.
\end{lemma}
\begin{proof}
Let $Y(\overline{y})=\Spec(\sO_{Y,\overline{y}})$ denote the strict
henselization of $Y$ at $\overline{y}$.
We have that $(E_{X/Y,\et})_{\overline{y}}=\varinjlim_{U} E_{X/Y}(U)$
where the limit is over all \etale{} neighborhoods $U\to Y$ of $\overline{y}$.
The induced map $(E_{X/Y,\et})_{\overline{y}}\to
E_{X/Y}\bigl(Y(\overline{y})\bigr)$ is a bijection since the functor $E_{X/Y}$
is locally of finite presentation. The latter set equals
$|X_{\overline{y}}|\cup \{\emptyset\}$ by
Lemma~\pref{L:unramified-over-strictly-henselian} (ii).
\end{proof}

\begin{lemma}\label{L:E-is-loc-constructible}
The sheaf $E_{X/Y}$ is \emph{locally constructible}, i.e., for any scheme $T$
and morphism $\map{\pi}{T}{Y}$, there is a natural isomorphism
$\pi^{-1}{E_{X/Y,\et}}\to E_{X\times_Y T/T,\et}$.
\end{lemma}
\begin{proof}
There is a natural transformation $E_{X/Y,\et}\to \pi_{*}{E_{X\times_Y
T/T,\et}}$ and hence by adjunction a natural transformation
$\map{\varphi}{\pi^{-1}{E_{X/Y,\et}}}{E_{X\times_Y T/T,\et}}$. It is enough to
verify that $\varphi$ is an isomorphism on geometric points. This follows from
Lemma~\pref{L:points-of-small-E}.
\end{proof}

\begin{proposition}\label{P:etale-envelope-is-representable}
The sheaf $E_{X/Y}$ is an algebraic space and the natural morphism
$\map{e}{E_{X/Y}}{Y}$ is \etale{} and representable.
\end{proposition}
\begin{proof}
Indeed, this statement is equivalent to Lemma~\pref{L:E-is-loc-constructible},
cf.~\cite[Exp.\ IX, pf.\ Prop.\ 2.7]{sga4} or~\cite[Ch.~V,
Thm.~1.5]{milne_etale_coh}. The space $E_{X/Y}$ is of finite presentation over
$Y$ if and only if the sheaf $E_{X/Y}$ is constructible.
\end{proof}

\begin{remark}
The algebraicity of $E_{X/Y}$ can also be shown as follows (and this is
essentially the method used in the following section). The question is local on
$Y$ so we can assume that $Y$ is affine and choose a
diagram~\eqref{E:local-embedding} as in the beginning of the introduction. It
can then be shown that there is an \etale{} representable and surjective
morphism $E_{X'/Y'}\to E_{X/Y}$ and that $E_{X'/Y'}$ is represented by the
scheme given as the gluing of two copies of $Y'$ along $Y'\setminus X'$.
Lemmas~\pref{L:etale-envelope-is-loc-of-fp}--\pref{L:E-is-loc-constructible}
are corollaries of this result and we do not need to use
Appendix~\ref{A:limits}.
%
% \pref{L:unramified-over-strictly-henselian}
% \pref{L:points-of-small-E}
\end{remark}

The distinguished section of $E_{X/Y}(Y)$, corresponding to $W=\emptyset$,
gives a section $j$ of $\map{e}{E_{X/Y}}{Y}$. As the diagonal of
$\map{f}{X}{Y}$ is open, we have a morphism $\map{i}{X}{E_{X/Y}}$ corresponding
to the diagonal $\{X\to X\times_Y X\}\in E_{X/Y}(X)$.

\begin{lemma}\label{L:etale-envelope-satisfies-conditions}
The morphism $\map{i}{X}{E_{X/Y}}$ is a closed immersion and $E_{X/Y}\setminus
i(X)=j(Y)$.
\end{lemma}
\begin{proof}
Let $T$ be a $Y$-scheme and let $\map{g}{T}{E_{X/Y}}$ be a morphism. To show
that $i$ is a closed immersion, it is enough to show that the pull-back of $i$
along $g$ is a closed immersion. Let $\map{w}{W}{X\times_Y T}$ be the open
immersion corresponding to $g$ so that $\map{\pi_2\circ w}{W}{T}$ is a closed
immersion. Then the squares
$$\vcenter{%
\xymatrix{W\ar[r]^{\pi_1\circ w}\ar@{(->}[d]^{\pi_2\circ w} & X\ar[d]^{i}\\
T\ar[r]^g & E_{X/Y}}}\quad\text{and}\quad
\vcenter{%
\xymatrix{T\setminus W\ar[r]\ar[d] & Y\ar[d]^{j}\\
T\ar[r]^g & E_{X/Y}}}$$
are commutative. The verification that these squares are cartesian is
straight-forward.
\end{proof}

% For good measure, we show that the functor $E_{X/Y}$ is the unique solution
% to our problem:

\begin{lemma}\label{L:etale-envelope-is-unique}
The triple $\map{e}{E_{X/Y}}{Y}$, $\map{i}{X}{E_{X/Y}}$, $\map{j}{Y}{E_{X/Y}}$,
is determined up to unique isomorphism by the condition that $E_{X/Y}\setminus
i(X)=j(Y)$.
\end{lemma}
\begin{proof}
Let $\map{e'}{E'}{Y}$, $\map{i'}{X}{E'}$ and $\map{j'}{Y}{E'}$ be another
triple of an \etale{} morphism, a closed immersion and an open immersion such
that $E'\setminus i'(X)=j'(Y)$. There is only one possible morphism
$\map{\varphi}{E'}{E_{X/Y}}$ such that $i=\varphi\circ i'$ and $j=\varphi\circ
j'$, since the graph of $\varphi$ --- an open subset of $E'\times_Y E_{X/Y}$
--- would be given as the union of the images of $\map{(i',i)}{X}{E'\times_Y
E_{X/Y}}$ and $\map{(j',j)}{Y}{E'\times_Y E_{X/Y}}$.

The graph of the map $i'$ determines an element of $E_{X/Y}(E')$, i.e., a
morphism $\map{\varphi}{E'}{E_{X/Y}}$, such that $i=\varphi\circ i'$ and
$j=\varphi\circ j'$. As $\varphi$ is a bijective \etale{} monomorphism, it is
an isomorphism.
\end{proof}

\begin{proof}[Proof of Theorem~\pref{T:main-theorem} (representable case)]
We postpone the proof of the existence and uniqueness of $E_{X/Y}$ for
non-representable morphisms $\map{f}{X}{Y}$ to the following
section. Similarly, for now, we only prove the functorial properties
\ref{TI:etale-pushfwd} and \ref{TI:closed-pullback} in the representable case.

The existence of $\map{e}{E_{X/Y}}{Y}$, $i$ and $j$ with the required
properties, for an unramified morphism $\map{f}{X}{Y}$ of algebraic spaces,
follows from Proposition~\pref{P:etale-envelope-is-representable} and
Lemma~\pref{L:etale-envelope-satisfies-conditions}. The triple $(e,i,j)$ is
unique with these properties by Lemma~\pref{L:etale-envelope-is-unique}.
That the triple commutes with base change follows from the uniqueness or
directly from the functorial description.

If $Y$ is an algebraic stack and $\map{f}{X}{Y}$ is a representable unramified
morphism, then we construct the representable and \etale{} morphism $E_{X/Y}\to
Y$ locally on $Y$~\cite[Ch.~14]{laumon}. We can also treat $E_{X/Y}$ as a
cartesian lisse-\etale{} sheaf of sets on $Y$.

This settles~\ref{TI:uniqueness} and~\ref{TI:base-change} in the representable
case. \ref{TI:iso} is trivial.
\ref{TI:sep} If $E_{X/Y}\to Y$ is separated then $j$ is closed and $i$ is open
and it follows that $f$ is \etale{} and separated. If $f$ is \etale{} then
$E_{X/Y}=X\amalg Y$ and $E_{X/Y}\to Y$ is separated if and only if $f$
is separated.
\ref{TI:uc,qc,repr} That $E_{X/Y}\to Y$ is universally
closed (resp.\ quasi-compact, resp.\ representable) if and only if $f$ is so,
follows from the fact that $i$ is a closed immersion and that $i\amalg j$ is
a surjective monomorphism (hence stabilizer preserving).

\ref{TI:qs,finite-pres} If $\map{e}{E_{X/Y}}{Y}$ is quasi-separated
then $j$ is quasi-compact so that $i$ is of constructible finite type
by Proposition~\pref{P:closed-imm-constructible}. It follows that $f=e\circ
i$ is quasi-separated and locally of constructible finite type. Conversely, if
$f$ is quasi-separated and locally of constructible finite type, then so is $i$
by Proposition~\pref{P:loc-constructible-type}. Hence $j$ is quasi-compact and,
a fortiori, so is $\map{i\amalg j}{X\amalg Y}{E_{X/Y}}$. As $f\amalg
\id{Y}=e\circ (i\amalg j)$ is quasi-separated it follows that
$\map{e}{E_{X/Y}}{Y}$ is quasi-separated. Finally, note that $e$ is
finitely presented if and only if $e$ is quasi-compact and quasi-separated
and that $f$ is of constructible finite type if and only if $f$ is
quasi-compact, quasi-separated and locally of constructible finite type.

\ref{TI:etale-pushfwd} and \ref{TI:closed-pullback} (representable case) Let
$\map{g}{V}{X}$ be \etale{} (resp.\ a closed immersion). We will construct a
morphism $\map{g_*}{E_{V/Y}}{E_{X/Y}}$ (resp.\ $\map{g^*}{E_{X/Y}}{E_{V/Y}}$)
using the functorial description.

In the \etale{} case, an element of $E_{V/Y}(T)$ corresponding to an open
subspace $W\subseteq V\times_Y T$ is mapped to the element corresponding to the
composition $W\to V\times_Y T\to X\times_Y T$. This composition, a priori only
\etale{}, is an open immersion since $W\to T$ is a closed immersion. That the
pull-back of $i_X$ (resp.\ $j_X$) along $g_*$ is $i_V$ (resp.\ $j_V$) is easily
verified. If $g$ is an open immersion, then $g_*$ is a monomorphism and hence
an open immersion.

In the case of a closed immersion, an element of $E_{X/Y}(T)$ corresponding to
an open subspace $W\subseteq X\times_Y T$ is mapped to the pull-back
$g_T^{-1}W\subseteq V\times_Y T$. If $\map{y}{\Spec(k)}{Y}$ is a point, then
the morphism $\map{g^*_y}{E_{X_y/y}=X_y\cup \{y\}}{E_{V_y/y}=V_y\cup \{y\}}$
is an isomorphism over the open and closed subscheme $V_y\cup \{y\}$ and maps
$X_y\setminus V_y$ onto the distinguished point $y$. It follows that
$i_V=g^*\circ i_X\circ g$, that $j_V=g^*\circ j_X$, that $g^*$ is surjective
and that $g^*$ is a monomorphism if and only if $g$ is bijective.
%%
%% It follows that $g^*$ is surjective. The morphism $g^*$ is thus an
%% isomorphism if and only if $g^*$ is a monomorphism. This happens if and only
%% if $g$ is bijective, i.e., if and only if $g$ is a nil-immersion.
% Indeed, look at the bijection $V\amalg Y\to X\amalg Y\to E_{X/Y}\to E_{V/Y}$.
% As the second map is a bijection, the first is a bijection iff the third is.

\ref{TI:loc-immersion} If $\map{e}{E_{X/Y}}{Y}$ is a local isomorphism, then
$f=e\circ i$ is a local immersion. Conversely, assume that $f$ is a local
immersion. The question whether $e$ is a local isomorphism is Zariski-local on
$E_{X/Y}$ and $Y$. We can thus, using~\ref{TI:etale-pushfwd}, assume that $f$
is a closed immersion. Then $E_{X/Y}=Y\cup_{Y\setminus X} Y\to Y$ is a local
isomorphism.
\end{proof}
\end{section}

%%%%%%%%%%%%%%%%%%%%%%%%%%%%%%%%%%%%%%%%%%%%%%%%%%%%%%%%%%%%%%%%%%%%%%%%%%%%%

\begin{section}{The general case}\label{S:general-case}
In this section we prove Theorem~\pref{T:main-theorem} for general
unramified morphisms of stacks.

\begin{definition}\label{D:etale-envelope:general}
If $\map{f}{X}{Y}$ is any (not necessarily representable) unramified morphism,
then we define a stack $E_{X/Y}$ over $\Sch_{/Y}$ (with the \etale{} topology)
as follows. The objects of the category $E_{X/Y}$ are $2$-commutative diagrams
%
%$$\xymatrix{W\ar[r]^p\ar@{(->}[d]^q & X\ar[d]\\
%T\ar[r] & Y\ar@{}[ul]|\circ}$$
%
$$\xymatrix{W\ar[r]^p\ar@{(->}[d]^q\drtwocell<\omit>{\varphi} & X\ar[d]\\
T\ar[r] & Y}$$
such that $T$ is a scheme, $\map{(p,\varphi,q)}{W}{X\times_Y T}$ is \etale{}
and $q$ is a closed immersion. Morphisms $(p',\varphi',q')\to (p,\varphi,q)$
are $2$-commutative diagrams
%
%$$\xymatrix{W'\ar[r]\ar@/^1pc/[rr]^>>>{p'}\ar@{(->}[d]^{q'}
%  & W\ar[r]_p\ar@{(->}[d]^q & X\ar[d]\\
%T'\ar[r]\ar@/_1pc/[rr] & T\ar[r]\ar@{}[ul]|\square & Y\ar@{}[ul]|\circ}$$
%
$$\xymatrix{W'\ar[r]\rruppertwocell<8>^{p'}{<-2.5>}\ar@{(->}[d]^{q'}\drtwocell<\omit>
  & W\ar[r]_p\ar@{(->}[d]^q\drtwocell<\omit>{\varphi} & X\ar[d]\\
  T'\ar[r]\rrlowertwocell<-8>{<2.2>} & T\ar[r] & Y}$$
such that the left square is $2$-cartesian and the pasting of the diagram is
$\varphi'$. The functor $E_{X/Y}\to \Sch_{/Y}$ is the functor mapping the
diagrams above onto their bottom rows. By \etale{} descent, the category
$E_{X/Y}$, which is fibered in groupoids, is a stack in the \etale{} topology.
\end{definition}

\begin{lemma}
Let $\injmap{q}{W}{T}$ be a closed immersion and let $Z\to W$ be an \etale{}
morphism of stacks. Then $q_*Z\to T$ is \etale{}. If $Z\to W$ is representable
(resp.\ surjective, resp.\ an open immersion) then so is $q_*Z\to T$. Here
$q_*Z$ denotes the stack over $\Sch_{/T}$ which associates to a scheme $T'\in
\Sch_{/T}$ the groupoid $\catHom_W(W\times_T T',Z)$.
\end{lemma}
\begin{proof}
The question is fppf-local on $T$ and we can thus assume that $T$ is a scheme.
Then $Z$ is Deligne--Mumford and we can pick an \etale{} presentation
$U\to Z$. It is enough to show that $q_*U\to q_*Z$ and $q_*U\to T$ are
\etale{} and representable and that the first map is surjective. We can thus
assume that $Z\to W$ is representable. Then $Z$ is a locally constructible
sheaf and it follows that $q_*Z$ is locally constructible by the proper base
change theorem, i.e., $q_*Z\to T$ is \etale{} and representable.

If $Z\to W$ is surjective, then so is $q_*Z\to T$. Indeed, this can be checked
on stalks. Let $t\in T$ be a point. If $t\in W$, then
$(q_*Z)_{\overline{t}}=Z_{\overline{t}}\neq
\emptyset$. If $t\notin W$, then $(q_*Z)_{\overline{t}}=Z(\emptyset)$ is the
final object --- the one-point set.

If $Z\to W$ is an open immersion, then $q_*Z=T\setminus (W\setminus Z)$ as
can be checked by passing to fibers.
\end{proof}

\begin{lemma}\label{L:etale-pushfwd}
Let $\map{g}{V}{X}$ be an \etale{} morphism. Then there is a natural \etale{}
morphism $\map{g_*}{E_{V/Y}}{E_{X/Y}}$. If $g$ is representable (resp.\
surjective, resp.\ an open immersion) then so is $g_*$.
\end{lemma}
\begin{proof}
This is similar to the proof of Theorem~\pref{T:main-theorem}
\ref{TI:etale-pushfwd} in
the representable case. Let $\xi\in E_{V/Y}$ be an object corresponding to
morphisms $\map{p}{W}{V}$, $\injmap{q}{W}{T}$. We let $g_*(\xi)\in E_{X/Y}$ be
the object corresponding to $g\circ p$ and $q$. On morphisms $g_*$ is defined
in the obvious way.

Let $T\to E_{X/Y}$ be a morphism corresponding to morphisms $\map{p}{W}{X}$ and
$\injmap{q}{W}{T}$. If $T'$ is a $T$-scheme, then the $T'$-points of
the pull-back $E_{V/Y}\times_{E_{X/Y}} T\to T$ is the groupoid of liftings of
$\map{p'}{W\times_T T'}{X}$ over $\map{g}{V}{X}$, or equivalently, the groupoid
of sections of $V\times_X W\times_T T'\to W\times_T T'$. This description is
compatible with pull-backs so that $E_{V/Y}\times_{E_{X/Y}} T$ is the stack
$q_*(V\times_X W)$ which is algebraic and \etale{} over $T$ by the previous
lemma. Moreover, if $V\to X$ is representable (resp.\ surjective, resp.\ an
open immersion) then so are $q_*(V\times_X W)\to T$ and $E_{V/Y}\to E_{X/Y}$.%
\end{proof}

\begin{lemma}\label{L:etale-envelope-is-algebraic}
The stack $E_{X/Y}$ is algebraic.
\end{lemma}
\begin{proof}
Let $Y'\to Y$ be a smooth presentation. Then $E_{X\times_Y Y'/Y'}\to E_{X/Y}$
is representable, smooth and surjective. Replacing $X$ and $Y$ with $X\times_Y
Y'$ and $Y'$ respectively, we can thus assume that $Y$ is a scheme.

Since $X\to Y$ is unramified, we have that $X$ is a Deligne--Mumford stack. Let
$V\to X$ be an \etale{} presentation. By Lemma~\pref{L:etale-pushfwd}, there is
an \etale{} representable surjection $E_{V/Y}\to E_{X/Y}$ and by
Proposition~\pref{P:etale-envelope-is-representable}, $E_{V/Y}$ is an algebraic
space. This shows that $E_{X/Y}$ is algebraic.
\end{proof}

\begin{proof}[Proof of Theorem~\pref{T:main-theorem} (general case)]
We have already proved that $E_{X/Y}$ is algebraic in
Lemma~\pref{L:etale-envelope-is-algebraic} and as in the representable case, we
can define morphisms $\map{i}{X}{E_{X/Y}}$ and $\map{j}{Y}{E_{X/Y}}$. That $i$
is a closed immersion and $j$ is an open immersion such that $j(Y)$ is the
complement of $i(X)$ follows exactly as in the proof of
Lemma~\pref{L:etale-envelope-satisfies-conditions}.

The uniqueness (which is up to unique $2$-isomorphism) of $E_{X/Y}$, $i$ and
$j$ satisfying $E_{X/Y}\setminus i(X)=j(Y)$ follows as in the proof of
Lemma~\pref{L:etale-envelope-is-unique} (because any morphism $E\to E_{X/Y}$
commuting with $i$ and $j$ is representable).

\ref{TI:etale-pushfwd} is Lemma~\pref{L:etale-pushfwd} and
\ref{TI:closed-pullback} follows exactly as in the representable case.
\end{proof}
\end{section}

% -------------------------------------------------------------------

\begin{section}{Applications}\label{S:applications}
There are two important consequences of Theorem~\pref{T:main-theorem}. The
first is that the classical description of unramified morphisms as
\etale{}-local embeddings remains valid when the target is not necessary
Deligne--Mumford. The second is that we obtain a \emph{canonical} factorization
of an unramified morphism into a closed immersion and an \etale{} morphism.
The following example illustrates the first consequence.

\begin{example}
It can be shown that if $X\to Y$ is an \etale{}, finitely presented and
representable morphism or a closed immersion of stacks and $\widetilde{X}\to X$
is a blow-up, then there exists a blow-up $\widetilde{Y}\to Y$ and an
$X$-morphism $\widetilde{Y}\times_Y X\to \widetilde{X}$. The analogous result
for a representable unramified morphism $X\to Y$ of constructible finite type
(e.g., of finite presentation) then follows from the existence of the \etale{}
envelope.
\end{example}

In the remainder of the section we outline an application where the canonicity
of the \etale{} envelope is crucial. It is shown
in~\cite{rydh_submersion_and_descent} that quasi-compact universally subtrusive
morphisms (e.g., universally submersive morphisms between noetherian spaces)
are morphisms of effective descent for the fibered category of finitely
presented \etale{} morphisms. Using Theorem~\pref{T:main-theorem} we obtain a
similar effective descent statement for unramified morphisms.

\begin{notation}
Let $\map{g}{S'}{S}$ be a morphism of algebraic spaces. Let $S''=S'\times_S S'$
be the fiber product and let $\map{\pi_1,\pi_2}{S''}{S'}$ be the two
projections.
\end{notation}

\begin{proposition}[Descent]
Let $\map{g}{S'}{S}$ be universally submersive. Let $X\to S$ and $Y\to S$ be
unramified morphisms of algebraic spaces. Then the sequence
$$\xymatrix{
\Hom_S(X_\red,Y_\red)\ar[r]^{g^*}
 & \Hom_{S'}(X'_\red,Y'_\red)\ar@<.5ex>[r]^{\pi_1^*}\ar@<-.5ex>[r]_{\pi_2^*}
 & \Hom_{S''}(X''_\red,Y''_\red)}$$
is exact. Here $X'$ and $Y'$ are the pull-backs of $X$ and $Y$ along $S'\to
S$, and $X''$ and $Y''$ are the pull-backs of $X$ and $Y$ along $S''\to S$.
\end{proposition}
\begin{proof}
A morphism $\map{f}{X_\red}{Y_\red}$ corresponds to an open subspace
$\Gamma\subseteq X_\red\times_S Y_\red$ such that the projection $\Gamma\to
X_\red$ is an isomorphism. Equivalently, since $Y\to S$ is unramified, an open
subset $\Gamma\subseteq |X\times_S Y|$ corresponds to a morphism $X_\red\to
Y_\red$ if and only if $\Gamma_\red\to X_\red$ is universally injective,
surjective and proper. As $g$ is surjective, it follows that
$\Hom_S(X_\red,Y_\red)\to\Hom_{S'}(X'_\red,Y'_\red)$ is injective.

Now if $\Gamma'\subseteq |X'\times_{S'} Y'|$ is an open subset such that
$\pi_1^{-1}\Gamma'=\pi_2^{-1}\Gamma'$ as subsets of $|X''\times_{S''} Y''|$,
then $\Gamma'$ is the pull-back of an open subset $\Gamma\subseteq |X\times_S
Y|$ since $g$ is universally submersive. If in addition $\Gamma'$ corresponds
to a morphism $X'_\red\to Y'_\red$, then $\Gamma'_\red\to X'_\red$ is
universally injective, surjective and proper. As $g$ is universally submersive,
it follows that $\Gamma_\red\to X_\red$ also is universally injective,
surjective and proper. Thus $\Gamma$ corresponds to a morphism $X_\red\to
Y_\red$ lifting $X'_\red\to Y'_\red$.
\end{proof}

\begin{theorem}[Effective descent]
Let $\map{g}{S'}{S}$ be a quasi-compact and quasi-separated universally
subtrusive morphism of algebraic spaces. Let $X'\to S'$ be an unramified
morphism of constructible finite type (e.g., of finite presentation)
of algebraic
spaces equipped with a ``reduced descent datum'' relative to $S'\to S$, i.e.,
an isomorphism $\map{\theta}{(\pi_1^*X')_\red}{(\pi_2^*X')_\red}$ satisfying
the usual cocycle condition after passing to reductions.
Then there is a unique unramified morphism $X\to S$ of constructible finite
type
and a schematically dominant morphism $X'\to X$ such that $X'\to X\times_S S'$
is a nil-immersion.
\end{theorem}
\begin{proof}
Let $X''_i=\pi_i^{*}X'$ for $i=1,2$ so that $X'':=(X''_1)_\red\iso
(X''_2)_\red$.  Consider the \etale{} envelopes $E_{X'/S'}$, $E_{X''/S''}$ and
$E_{X''_i/S''}$.  The nil-immersions $X''\inj X''_i$ induce natural
isomorphisms $E_{X''_i/S''}\to E_{X''/S''}$. As the \etale{} envelope commutes
with pull-back, there is a canonical isomorphism $E_{X''/S''}\iso
\pi_1^{*}E_{X'/S'}\iso \pi_2^{*}E_{X'/S'}$ which equips $E_{X'/S'}$ with
a descent datum.

The morphism $E_{X'/S'}\to S'$ is \etale{} and of finite presentation. Thus, it
descends to a morphism $E\to S$ which is \etale{} and of finite
presentation~\cite[Thm.~5.17]{rydh_submersion_and_descent}. The induced
morphism
$\map{h}{E_{X'/S'}}{E}$ is a pull-back of $g$ and thus universally subtrusive.
As $h$ is surjective and $\pi_1^{-1}(i'(X'))=\pi_2^{-1}(i'(X'))$ as sets, there
is a unique subset $X\subseteq E$ such that $h^{-1}(X)=i'(X')$. Since $h$ is
subtrusive and $i'(X')\subseteq E_{X'/S'}$ is closed and constructible, it
follows that $X$ is closed and constructible. We consider the set $X$ as a
closed subspace of $E$ by taking the ``schematic image'' of $X'\inj
E_{X'/S'}\to E$. Then $X\to S$ satisfies the conditions of the theorem.
\end{proof}

\newcommand{\catUnramConst}{\mathbf{Unr}_{\mathrm{cons}}}

\begin{corollary}
Let $\catUnramConst(S)$ be the category of unramified morphisms
$X\to S$ of constructible finite type with $X$ reduced and let $\catUnramConst(S'\to S)$ be the
category of unramified morphisms $X'\to S'$, of constructible finite type, equipped with a
reduced descent datum and with $X'$ reduced. There is a natural functor
$\catUnramConst(S)\to \catUnramConst(S'\to S)$ taking $X\to S$ to $(X\times_S
S')_\red\to S'$ and the induced descent datum. This functor is an equivalence
of categories.
\end{corollary}

\end{section}

% -------------------------------------------------------------------
\appendix
% -------------------------------------------------------------------

\begin{section}{Algebraic spaces and stacks}\label{A:stacks}
A sheaf of sets $F$ on the category of schemes $\Sch$ with the \etale{}
topology is an \emph{algebraic space} if there exists a scheme $X$ and a
morphism $X\to F$ which is represented by surjective \etale{} morphisms
of schemes~\cite[D\'ef.~5.7.1]{raynaud-gruson}, i.e., for any scheme $T$
and morphism $T\to F$, the fiber product $X\times_F T$ is a scheme and
$X\times_F T\to T$ is surjective and \etale{}.

A \emph{stack} is a category fibered in groupoids over $\Sch$ with the
\etale{} topology satisfying the usual sheaf condition~\cite{laumon}.
A morphism $\map{f}{X}{Y}$ of stacks is \emph{representable} if for any scheme
$T$ and morphism $T\to Y$, the $2$-fiber product $X\times_Y T$ is an algebraic
space. A stack $X$ is \emph{algebraic} if there exists a smooth presentation,
i.e., a smooth, surjective and representable morphism $U\to X$ where $U$ is a
scheme. A stack $X$ is \emph{Deligne--Mumford} if there exists an \etale{}
presentation. A stack $X$ is Deligne--Mumford if and only if $X$ is algebraic
and the diagonal $\Delta_X$ is unramified. A morphism $\map{f}{X}{Y}$ of stacks
is \emph{quasi-separated} if the diagonal $\Delta_{X/Y}$ is quasi-compact and
quasi-separated, i.e., if both $\Delta_{X/Y}$ and its diagonal are
quasi-compact.

\begin{remark}[Quasi-separatedness]
We do not require that algebraic spaces and stacks are quasi-separated nor that
the diagonal of an algebraic stack is separated. The queasy reader may assume
that the diagonals of all stacks and algebraic spaces are separated and
quasi-compact (as in~\cite{knutson_alg_spaces,laumon}) but this is not
necessary in this paper. The reader should however note that unless we work
with noetherian stacks or finitely presented unramified morphisms, stacks and
algebraic spaces with non-quasi-compact diagonals will appear.

The diagonal of a (not necessarily quasi-separated) algebraic space is
representable by schemes. This follows by effective fppf-descent of
monomorphisms which are locally of finite type. Indeed, more generally the
class of locally quasi-finite and separated morphisms is an effective class in
the fppf-topology (cf.\ \cite[App.]{murre}, \cite[Exp.~X, Lem.~5.4]{sga3}
or~\cite[pf.\ of 5.7.2]{raynaud-gruson}).

The diagonal of an algebraic stack $X$ is representable. This follows
by~\cite[pf.\ of Prop.~4.3.1]{laumon} as~\cite[Cor.~1.6.3]{laumon} generalizes
to arbitrary algebraic spaces.

The characterization of Deligne--Mumford stacks as algebraic stacks with
unramified diagonal is valid for arbitrary algebraic stacks. Indeed,
the proof of~\cite[Thm.~8.1]{laumon} does not use that the diagonal
is separated and quasi-compact.
\end{remark}
\end{section}

% -------------------------------------------------------------------

\begin{section}{Unramified and \etale{} morphisms of stacks}\label{A:unramified-and-etale}
We use the modern terminology of unramified
morphisms~\cite{raynaud_hensel_rings}: an unramified morphism of schemes is a
formally unramified morphism which is \emph{locally of finite type} (and not
necessarily locally of finite presentation). Equivalently, an unramified
morphism is a morphism locally of finite type such that the diagonal is an open
immersion~\cite[17.4.1.2]{egaIV}. Recall that an \etale{} morphism of schemes
is a formally \etale{} morphism which is locally of finite presentation or
equivalently, a flat and unramified morphism which is locally of finite
presentation~\cite[17.6.2]{egaIV}. These definitions generalize to include
\emph{non-representable} morphisms as follows:

\begin{definition}
A morphism $\map{f}{X}{Y}$ of algebraic stacks is \emph{unramified} if $f$ is
locally of finite type and the diagonal $\Delta_f$ is \etale{}. A morphism
$\map{f}{X}{Y}$ of algebraic stacks is \emph{\etale} if $f$ is locally of
finite presentation, flat and unramified.
\end{definition}

For representable $f$ this definition of unramified agrees with the usual since
an \etale{} monomorphism is an open immersion~\cite[Thm.~17.9.1]{egaIV}.
The notions of unramified and \etale{} are fpqc-local on the target and
\etale-local on the source~\cite[2.2.11~(iv), 2.7.1, 17.7.3, 17.7.7]{egaIV}.

\begin{proposition}\label{P:etale=smooth+unramified}
Let $\map{f}{X}{Y}$ be a morphism of algebraic stacks. The following are
equivalent:
\begin{enumerate}
\item $f$ is \etale{}.
\item $f$ is smooth and unramified.
\end{enumerate}
\end{proposition}
\begin{proof}
As a smooth morphism is flat and locally of finite presentation (ii) implies
(i). To see that (i) implies (ii), take a smooth presentation $U\to X$.
If $f$ is \etale{} then $U\times_X U\to U\times_Y U$ is \etale{}. Thus, the
projections $U\times_Y U\to U$ are smooth at the points in the image of
$U\times_X U$. Since $U\times_X U\to U$ is surjective and $U\to Y$ is flat,
it follows that $U\to Y$ is smooth by flat descent and, a fortiori, that
$X\to Y$ is smooth.
% \cite[Prop.~17.7.1]{egaIV}
\end{proof}

\begin{proposition}\label{P:unramified-char}
Let $\map{f}{X}{Y}$ be a morphism of algebraic stacks. The following are
equivalent:
\begin{enumerate}
\item $f$ is unramified.
\item $f$ is locally of finite type and for every point
$\Spec(k)\to Y$ we have that $X\times_Y \Spec(k)\to \Spec(k)$ is unramified.
\item $f$ is locally of finite type and for every point
$\Spec(k)\to Y$ we have that $X\times_Y \Spec(k)$ is geometrically reduced,
Deligne--Mumford and discrete.
%\item There exists an unramified morphism of schemes $X'\to Y'$ and faithfully
%flat morphisms $X'\to X$ and $Y'\to Y$, locally of finite presentation such
%that $X'\to X\times_Y Y'$ is \etale{} and surjective.
\end{enumerate}
\end{proposition}
\begin{proof}
% (i)$\iff$(iv): Since $\Delta_{X/Y}$ is \etale{}, we have that $f$ is
% relatively Deligne--Mumford. Thus, there exists a smooth presentation
% $Y'\to Y$ and an \etale{} presentation $X'\to X\times_Y Y'$. This shows that
% (i)$\implies$(iv). The converse follows from
% Properties~\pref{P:etale-fppf-locality-of-unramified/etale}.
Clearly (i)$\implies$(ii). If $f$ is representable, then it is well-known that
(ii)$\implies$(iii)$\implies$(i)~\cite[17.4.1.2]{egaIV}.
For general $f$, to see that (ii)$\implies$(iii) we can assume that
$Y=\Spec(k)$ so that $X$ is Deligne--Mumford. As both (ii) and (iii) are
\etale{}-local on $X$ we can also assume that $X$ is a scheme so that $f$ is
representable and (ii)$\implies$(iii) by the representable case.

If (iii) holds, then the fibers of the diagonal are unramified and hence
$\Delta_f$ is unramified, i.e., $f$ is Deligne--Mumford. Let $Y'\to
Y$ be a smooth presentation and let $X'\to X\times_Y Y'$ be an \etale{}
presentation. Then the representable morphism $X'\to Y'$ also satisfies
condition (iii) and hence is unramified. This shows that (iii)$\implies$(i).
\end{proof}

In the remainder of this section we will show that the definitions of
unramified and \etale{} given above have a more standard formal description.

\begin{definition}
Let $S$ be a stack and let $X$ and $Y$ be stacks over $S$. We let
$\catHom_S(X,Y)$ be the groupoid with objects $2$-commutative diagrams
$$\xymatrix{X\ar[r]^f\rrlowertwocell{<2>\tau} & Y\ar[r] & S}$$
and morphisms $\map{\varphi}{(f_1,\tau_1)}{(f_2,\tau_2)}$, $2$-commutative
diagrams
$$\vcenter{\xymatrix{X\rtwocell<3>^{f_1}_{f_2}{\varphi}\rrlowertwocell<-12>{<3>\tau_2}
 & Y\ar[r] & S}} \quad \text{such that $\tau_2\circ\varphi=\tau_1$.}$$
\end{definition}

% $\catHom_S(X,Y)\to \catHom_X(X,X\times_S Y)$ is an equivalence of categories.
We note that if $Y\to S$ is representable, then the groupoid $\catHom_S(X,Y)$
is equivalent to a set.

\begin{definition}
Let $\map{f}{X}{Y}$ be a morphism of stacks. We say that $f$ is \emph{formally
unramified} (resp.\ \emph{formally Deligne--Mumford}, resp.\ 
\emph{formally smooth}, resp.\ \emph{formally \etale{}}) if for every
% (affine?)
$Y$-scheme $T$ and every closed subscheme $T_0\inj T$ defined by a nilpotent
ideal sheaf the functor
$$\catHom_Y(T,X)\to \catHom_Y(T_0,X)$$
is fully faithful (resp.\ faithful, resp.\ essentially surjective, resp.\ an
equivalence of categories).
\end{definition}

\begin{remark}
The functor $\catHom_Y(T,X)\to \catHom_Y(T_0,X)$ is essentially surjective if
and only if for every $2$-commutative diagram
$$\xymatrix{T_0\ar[r]\ar@{(->}[d]\drtwocell<\omit>{\tau} & X\ar[d]\\
T\ar[r] & Y}$$
there exists a morphism $T\to X$ and a $2$-commutative diagram
$$\vcenter{\xymatrix{T_0\ar[r]\ar@{(->}[d]\rtwocell<\omit>{<1.5>\varphi} & X\ar[d]\\
T\ar[r]\ar[ur]\rtwocell<\omit>{<-2>\psi} & Y}}
\quad \text{such that $\tau=\psi\circ\varphi$.}$$
If $\map{f}{X}{Y}$ is locally of finite presentation, then it can be shown that
it is enough to consider strictly henselian $T$ and closed subschemes $T_0\inj
T$ defined by a square-zero ideal, cf.~\cite[Prop.~4.15 (ii)]{laumon}.
%
% NB! It is important to require that $\tau=\psi\circ\varphi$, not only
% the existence of a 2-commutative diagram (as written in Laumon). Consider
% the example $X=T_0=\Spec(k)$, $Y=BG$ and $T=\Spec(k[e]/e^2)$. There is always
% a 2-commutative diagram but with a specified $\tau$, there is only a
% 2-commutative diagram if every 2-isomorphism over $k$ lifts to $k[e]/e^2$,
% e.g., if $G$ is smooth.
\end{remark}

Formally unramified (resp.\ \dots) morphisms are stable under base change,
products and composition, cf.~\cite[Prop.~17.1.3]{egaIV}.

\begin{proposition}\label{P:form-unram/etale-and-diagonal}
Let $\map{f}{X}{Y}$ be a morphism of stacks. Then $f$ is formally unramified
(resp.\ formally Deligne--Mumford) if and only if the diagonal $\Delta_f$
is formally \etale{} (resp.\ formally unramified).
\end{proposition}
\begin{proof}
Let $T$ be a $Y$-scheme and let $\injmap{j}{T_0}{T}$ be a closed subscheme
defined by a
nilpotent ideal. Let $(f_1,\tau_1)$ and $(f_2,\tau_2)$ be two objects of
$\catHom_{Y}(T,X)$. This determines a morphism
$\map{F=(f_1,\tau_2^{-1}\circ\tau_1,f_2)}{T}{X\times_Y X}$. Conversely, a
morphism $\map{F}{T}{X\times_Y X}$ gives rise to a (non-unique) pair
$(f_1,\tau_1),(f_2,\tau_2)$ of objects in $\catHom_{Y}(T,X)$ such that
$F=(f_1,\tau_2^{-1}\circ\tau_1,f_2)$.

Fix a pair of objects $(f_1,\tau_1),(f_2,\tau_2)$ and a morphism
$\map{F}{T}{X\times_Y X}$ as above. As the diagonal of $f$ is representable,
the groupoid $\catHom_{X\times_Y X}(T,X)$ is equivalent to the set
$\Hom_{X\times_Y X}(T,X):=\pi_0 \catHom_{X\times_Y X}(T,X)$.
There is a natural bijection between the set of $2$-morphisms
$\Hom(f_1,f_2)$ and the set $\Hom_{X\times_Y X}(T,X)$. Thus
$\Hom(f_1,f_2)\to \Hom(f_1\circ j,f_2\circ j)$ is bijective (resp.\ injective)
if and only if $\Hom_{X\times_Y X}(T,X)\to\Hom_{X\times_Y X}(T_0,X)$ is
bijective (resp.\ injective).
\end{proof}

\begin{corollary}
Let $\map{f}{X}{Y}$ and $\map{g}{Y}{Z}$ be two morphisms.
\begin{enumerate}
\item If $g\circ f$ is formally Deligne--Mumford then so is $f$.
\item If $g\circ f$ is formally unramified and $g$ is formally Deligne--Mumford,
then $f$ is formally unramified.
\end{enumerate}
\end{corollary}

\begin{corollary}
Let $\map{f}{X}{Y}$ be a morphism of stacks.
\begin{enumerate}
\item $f$ is smooth if and only if $f$ is locally of finite presentation
and formally smooth.\label{CI:smooth}
\item $f$ is \etale{} if and only if $f$ is locally of finite presentation
and formally \etale{}.\label{CI:etale}
\item $f$ is unramified if and only if $f$ is locally of finite type and
formally unramified.\label{CI:unramified}
\item $f$ is Deligne--Mumford (i.e., $\Delta_f$ is unramified) if
and only if $f$ is formally Deligne--Mumford.\label{CI:D-M}
\end{enumerate}
\end{corollary}
\begin{proof}
If $f$ is representable, then~\ref{CI:smooth}, \ref{CI:etale}
and~\ref{CI:unramified} are definitions and~\ref{CI:D-M} is trivial. For
general $f$, statement~\ref{CI:unramified} [resp.\ \ref{CI:D-M}] follows from
Proposition~\pref{P:form-unram/etale-and-diagonal} using
statement~\ref{CI:etale} [resp.\ \ref{CI:unramified}] for the representable
diagonal $\Delta_f$. Statement~\ref{CI:smooth} is~\cite[Prop.~4.15
(ii)]{laumon}. Finally~\ref{CI:etale} follows from
Proposition~\pref{P:etale=smooth+unramified} and~\ref{CI:smooth}
and~\ref{CI:unramified}.
%
% To show~\ref{CI:etale}, we note that both properties of $f$ are fpqc-local on
% the target and \etale{}-local on the source so we can pass to the
% representable case as $f$ is Deligne--Mumford.
\end{proof}

\end{section}

% -------------------------------------------------------------------

\begin{section}{Auxiliary limit results}\label{A:limits}
In this appendix we give some fairly standard limit results. For
simplicity we state these results for algebraic spaces although they
remain valid for algebraic stacks.

% Loc. of finite pres: ega IV 3, p. 51

%% \begin{definition}
%% Let $S$ be an algebraic stack and let
%% $\map{F}{\Sch_{/S}^{\op}}{\Set}$ be a functor. We say that $F$ is
%% \emph{locally of finite presentation} if for every inverse limit of affine
%% schemes $T=\varprojlim T_\lambda$ over $S$ we have that
%% %
%% $$\varinjlim_{\lambda} F(T_\lambda)\to F(T)$$
%% %
%% is bijective.
%% \end{definition}

%% \begin{proposition}[{\cite[Prop.~4.15 (i)]{laumon}}]
%% Let $\map{f}{X}{Y}$ be a representable morphism of algebraic stacks
%% and let $\map{h_{X/Y}=\Hom_Y(-,X)}{\Sch_{/Y}}{\Set}$. Then $f$ is locally
%% of finite presentation if and only if $h_{X/Y}$ is locally of finite
%% presentation.
%% \end{proposition}

\begin{proposition}\label{P:semi-std-limits}
Let $S_0$ be an algebraic space and let
$S=\varprojlim_{\lambda} S_\lambda$ be an inverse limit of algebraic spaces that
are affine over $S_0$.
Let $X$ be a \emph{quasi-compact and quasi-separated} algebraic space
and let $X\to S$ be a morphism locally of finite presentation. Then there
exists an index $\lambda$, a quasi-compact and quasi-separated algebraic space
$X_\lambda$, a morphism $X_\lambda\to S_\lambda$ locally of finite presentation
and an $S$-isomorphism $X_\lambda\times_{S_\lambda} S\to X$. If $X\to S$ is
\etale{} then it can be arranged so that $X_\lambda\to S_\lambda$ also is
\etale{}.
\end{proposition}
\begin{proof}
Since $X$ is quasi-compact, we can assume that $S_0$ is quasi-compact
after replacing $S_0$ by an open subspace. Let $V_0\to S_0$ be an
\etale{} presentation with $V_0$ an affine scheme. Let
$V_\lambda=V_0\times_{S_0} S_\lambda$ and $V=V_0\times_{S_0} S$. Finally
choose an affine scheme $U$ and an \etale{} morphism $U\to V\times_S X$
such that $U\to X$ is surjective. Note that $U\to X$ and $U\to V$ are of finite
presentation. Let $R=U\times_X U$ and note that $\map{j}{R}{U\times_S U}$ is a
monomorphism of finite presentation as $X$ is quasi-separated.

% For stacks we can use~\cite[Prop.~4.18]{laumon} which is ok since $V_0$
% is affine.
Since $U\to V$ and $j$ are of finite presentation, there is 
for sufficiently large $\lambda$ a finitely presented scheme
$U_\lambda\to V_\lambda$, a finitely presented monomorphism
$\map{j_\lambda}{R_\lambda}{U_\lambda\times_{S_\lambda} U_\lambda}$
and cartesian diagrams
$$\vcenter{\xymatrix{U\ar[r]\ar[d] & U_\lambda\ar[d]\\
V\ar[r] & V_\lambda\ar@{}[ul]|\square}}\quad\text{and}\quad
\vcenter{\xymatrix{R\ar[r]\ar[d]^{j} & R_\lambda\ar[d]^{j_\lambda}\\
U\times_S U\ar[r] & U_\lambda\times_{S_\lambda} U_\lambda\ar@{}[ul]|\square}}$$
such that $\map{s_\lambda,t_\lambda}{R_\lambda}{U_\lambda}$ are \etale{}
with $s_\lambda=\pi_1\circ j_\lambda$ and $t_\lambda=\pi_2\circ j_\lambda$,
and $R_\lambda$ is quasi-compact.
The morphism $j_\lambda=(s_\lambda,t_\lambda)$ defines an equivalence relation
if and only if
\begin{enumerate}
\item[(R)] the pull-back of $j_\lambda$ along
  $\map{\Delta_{U_\lambda}}{U_\lambda}{U_\lambda\times_{S_\lambda}
  U_\lambda}$ is an isomorphism,
\item[(S)] the pull-back of $j_\lambda$ along
  $\map{(t_\lambda,s_\lambda)}{R_\lambda}{U_\lambda\times_{S_\lambda}
  U_\lambda}$ is an isomorphism, and
\item[(T)] the pull-back of $j_\lambda$ along
  $\map{(s\circ \pi_1,t\circ \pi_2)}
{R_\lambda\times_{t_\lambda,U_\lambda,s_\lambda} R_\lambda}
{U_\lambda\times_{S_\lambda} U_\lambda}$ is an isomorphism.
\end{enumerate}
The pull-back of the above maps along
$U\to U_\lambda$, $R\to R_\lambda$ and
$R\times_U R\to R_\lambda\times_{U_\lambda} R_\lambda$ respectively are
isomorphisms since $j$ is an equivalence relation. Noting that $j_\lambda$ is of
finite presentation and $U_\lambda$, $R_\lambda$ and
$R_\lambda\times_{U_\lambda} R_\lambda$ are quasi-compact, we conclude that
$j_\lambda$ is an equivalence relation for sufficiently
large $\lambda$ by~\cite[Thm.~8.10.5 (i)]{egaIV}.
The quotient $X_\lambda$ of this equivalence relation is
a quasi-compact and quasi-separated algebraic space which is locally of finite
presentation over $S_\lambda$.
The last assertion follows from~\cite[Prop.~17.7.8 (ii)]{egaIV}.
\end{proof}

Note that Proposition~\pref{P:semi-std-limits} reduces to the standard limit
result on finitely presented objects if $S_0$ is quasi-compact and
quasi-separated.

\begin{proposition}\label{P:limit-of-ft+qs}
Let $S_0$ be an affine scheme and let $S=\varprojlim_\lambda S_\lambda$ be an
inverse limit of affine $S_0$-schemes. Let $X_0$ be an algebraic space and let
$\map{f_0}{X_0}{S_0}$ be of \emph{finite type and quasi-separated}. Let
$\map{f_\lambda}{X_\lambda}{S_\lambda}$ and $\map{f}{X}{S}$ denote the base
changes of $f_0$. Then $f$ is a monomorphism (resp.\ closed immersion) if and
only if $f_\lambda$ is a monomorphism (resp.\ closed immersion) for
sufficiently large $\lambda$.
\end{proposition}
\begin{proof}
The condition is clearly sufficient. To see that the condition is necessary
for the property ``monomorphism'', recall that a morphism $f$ is a
monomorphism if and only if its diagonal $\Delta_f$ is an isomorphism. As
the diagonal is strongly representable and finitely presented the necessity
in this case follows from~\cite[Thm.~8.10.5 (i)]{egaIV}. If $f$ is a closed
immersion then by the previous case $f_\lambda$ is a monomorphism for
sufficiently large $\lambda$. In particular $f_\lambda$ is quasi-finite and
separated so that $f_\lambda$ is strongly representable~\cite[Thm.~A.2]{laumon}
and Zariski's main theorem~\cite[Cor.~18.12.13]{egaIV} gives rise to
a factorization $X_\lambda\to Y_\lambda\to S_\lambda$ of $f_\lambda$ where the
first morphism is a quasi-compact open immersion and the second morphism is
finite. As $X\to Y_\lambda\times_{S_\lambda} S$ is an open and closed immersion
so is $X_\lambda\to Y_\lambda$ for sufficiently large $\lambda$.
In particular $X_\lambda\to S_\lambda$ is a proper monomorphism and hence a
closed immersion.
\end{proof}

More generally Proposition~\pref{P:limit-of-ft+qs} holds for properties such as:
proper, finite, affine, quasi-affine, separated; but not for other
properties such as being an isomorphism.
\end{section}

% -------------------------------------------------------------------

\begin{section}{Morphisms of constructible finite type}
\label{A:constructible}
In this section we define morphisms (locally) of \emph{constructible finite
type}. A morphism (locally) of finite presentation is (locally) of
constructible finite type and a morphism (locally) of constructible finite type
is (locally) of finite type. For morphisms of noetherian stacks, all these
notions coincide.

Let $X$ be a scheme. Recall that a subset $W\subseteq X$ is ind-constructible
(resp.\ pro-constructible) if locally $W$ is a union (resp.\ an intersection)
of constructible subsets~\cite[D\'ef.~7.2.2]{egaI_NE}. If $\map{p}{U}{X}$ is
locally of finite presentation and surjective, then $W$ is ind-constructible
(resp.\ pro-constructible, resp.\ constructible) if and only if $p^{-1}(W)$ is
so~\cite[Cor.~7.2.10]{egaI_NE}. Now let $X$ be an algebraic stack. We define a
subset $W\subseteq X$ to be ind-constructible (resp.\ pro-constructible, resp.\
constructible) if $p^{-1}(W)$ is so for some presentation $\map{p}{U}{X}$ with
$U$ a scheme. This definition does not depend on the choice of presentation.

\begin{definition}
Let $\map{f}{X}{Y}$ be a morphism of algebraic stacks. The morphism $f$ is
\emph{ind-constructible} if the image under $f$ of any ind-constructible subset
is ind-constructible. If this holds after arbitrary base change $Y'\to Y$, then
we say that $f$ is \emph{universally ind-constructible}.
\end{definition}

The primary example of an ind-constructible morphism is a morphism which is
locally of finite presentation~\cite[Prop.~7.2.3]{egaI_NE}.

\begin{definition}
A morphism $\map{f}{X}{Y}$ of stacks is \emph{locally of construct\-ible
finite type} if $f$ is locally of finite type and universally ind-constructible.
A morphism $f$ is of \emph{constructible finite type} if $f$ is quasi-compact,
quasi-separated and locally of constructible finite type.
\end{definition}

Morphisms (locally) of finite presentation are (locally) of constructible
finite type.
The image of a pro-constructible set under a quasi-compact morphism is
pro-constructible~\cite[Prop.~7.2.3]{egaI_NE}. It follows that a morphism 
of constructible finite type takes constructible subsets to constructible
subsets~\cite[Prop.~7.2.9]{egaI_NE}.

% Equivalently, $f$ is constructible if and only if $f$ is quasi-compact,
% quasi-separated and the image of any constructible subset is
% constructible. Indeed, any ind-constructible subset is a union of
% constructible subsets and since $f$ is quasi-compact, $f$ takes
% pro-constructible subsets onto pro-constructible subsets.

\begin{proposition}\label{P:loc-constructible-type}
Let $\map{f}{X}{Y}$ and $\map{g}{Y}{Z}$ be morphisms of algebraic stacks.
\begin{enumerate}
\item If $f$ and $g$ are locally of constructible finite type, then so is
$g\circ f$.
\label{PI:ic:1}
\item If $g\circ f$ is locally of constructible finite type and
if $g$ is locally of finite type, then $f$ is locally of constructible finite
type.
\label{PI:ic:2}
\end{enumerate}
\end{proposition}
\begin{proof}
\ref{PI:ic:1} is obvious.
\ref{PI:ic:2} As the diagonal of $g$ is locally of finite presentation, we have
that $f$ is the composition of a morphism locally of constructible finite type
and a morphism locally of finite presentation,
hence locally of constructible finite type.
\end{proof}

\begin{proposition}\label{P:closed-imm-constructible}
Let $\injmap{f}{Z}{X}$ be a closed immersion of algebraic stacks. The
following are equivalent:
\begin{enumerate}
\item $f$ is of constructible finite type.
\item The subset $|Z|\subseteq |X|$ is constructible.
\item The open immersion $X\setminus Z\to X$ is quasi-compact.
\end{enumerate}
\end{proposition}
\begin{proof}
Immediate from the fact that an open immersion is pro-constructible if and
only if it is quasi-compact~\cite[Prop.~7.2.3]{egaI_NE}.
\end{proof}

%% \begin{remark}
%% Let $\map{f}{X}{Y}$ be a morphism locally of finite type. Then $f$ is
%% locally of constructible finite type in the following cases:
%% %
%% \begin{enumerate}
%% \item $f$ is locally of finite presentation.
%% \item $f$ is a universal homeomorphism (e.g., any nil-immersion).
%% \item $f$ is an immersion such that $Y\setminus f(X)$ is
%% retro-compact (i.e., $f(X)$ is an ind-constructible locally closed subset).
%% \end{enumerate}
%% \end{remark}

Not every quasi-separated morphism of finite type is of constructible finite
type. For example, there are closed immersions which are not constructible.
A morphism locally of finite presentation, e.g., an \etale{} morphism, is
of constructible finite type if and only if it is of finite presentation.

Let $\map{f}{X}{Y}$ be an unramified morphism with a factorization $X\inj
X_1\to Y$ where $X\inj X_1$ is a nil-immersion and $X_1\to Y$ is unramified and
of finite presentation. Then $f$ is of constructible finite type. Conversely, if
$Y$ is quasi-compact and quasi-separated it is likely that every unramified
morphism $f$ of constructible finite type has such a factorization.
%
% Indeed, this holds if $Y$ is pseudo-noetherian, e.g., if $Y$ is
% Deligne--Mumford.
%
\end{section}

%%%%%%%%%%%%%%%%

\bibliography{unramified}
\bibliographystyle{dary}

\end{document}